\documentclass{llncs}
\usepackage{amsfonts}
\usepackage[utf8x]{inputenc}
\usepackage{amssymb,amsmath}

\usepackage{color}
\definecolor{keywordcolor}{rgb}{0.7, 0.1, 0.1}   
\definecolor{commentcolor}{rgb}{0.4, 0.4, 0.4}   
\definecolor{symbolcolor}{rgb}{0.0, 0.1, 0.6}    
\definecolor{sortcolor}{rgb}{0.1, 0.5, 0.1}      

\usepackage{listings}

\lstMakeShortInline"

\newcommand{\N}{\mathbb{N}}
\newcommand{\Z}{\mathbb{Z}}
\newcommand{\dol}{\mbox{\textdollar}}
\begin{document}

\title{A Lean formalization of Matiyasevi\v{c}'s Theorem}
\titlerunning{Matiyasevi\v{c}'s Theorem}  
%
\author{Mario Carneiro}
\authorrunning{Mario Carneiro} 
%
%
\institute{Carnegie Mellon University, Pittsburgh PA 15289, USA,\\
\email{mcarneir@andrew.cmu.edu}}

\maketitle              

\begin{abstract}
In this paper, we present a formalization of Matiyasevi\v{c}'s theorem, which states that the power function is Diophantine, forming the last and hardest piece of the MRDP theorem of the unsolvability of Hilbert's 10th problem. The formalization is performed within the Lean theorem prover, and necessitated the development of a small number theory library, including in particular the solution to Pell's equation and properties of the Pell $x,y$ sequences.
\keywords{Diophantine equations, Pell's equation, MRDP theorem, Matiyasevi\v{c} theorem, Lean, Computability}
\end{abstract}
\section{Introduction}
In 1900, David Hilbert presented a list of 23 unsolved problems to the International Congress of Mathematicians \cite{hilbert}. Of these, the tenth was the following: To find an algorithm to determine whether a given polynomial Diophantine equation with integer coefficients has an integer solution. This problem remained unsolved for 70 years before Matiyasevi\v{c} proved there is no such algorithm.

A Diophantine equation is an equation such as
$$x+xy+3z^2-1=0\qquad x,y,z\in\Z,$$
in which we seek to find the \emph{integer} solutions to a polynomial with integer coefficients. Solutions to Diophantine equations can encode a wide variety of arithmetic operations, and even some very advanced number theory: For example, for fixed $n>2$, if the Diophantine equation $$x^n+y^n-z^n=0$$ had a solution with $xyz\ne 0$, then Fermat's last theorem would be false. So a positive answer to Hilbert's problem would have been very big news indeed.

Define a Diophantine set as a subset $S\subseteq \N^k$ such that there exists an $m\in\N$ and an integer polynomial $p\in\Z[x_1,\dots,x_k,y_1,\dots,y_m]$ such that
$$(x_1,\dots,x_k)\in S\iff \exists y_1,\dots,y_m\in\N,\ p(x_1,\dots,x_k,y_1,\dots,y_m)=0.$$
Here the $x_i$ are viewed as ``parameters'' of the Diophantine equation, and the $y_i$ are the unknowns, or ``dummy variables''. Hilbert's problem asks if we can determine if a given Diophantine set is nonempty from this description of it.

The Matiyasevi\v{c}--Robinson--Davis--Putnam (MRDP) theorem asserts that a set $S$ is Diophantine if and only if it is recursively enumerable, that is, the elements of $S$ can be enumerated (with repetition allowed) as the range of a partial computable function. Since the problem of determining if a partial computable function is nonempty is equivalent to the halting problem, which is undecidable, this shows that there is no algorithm for Hilbert's tenth problem.

Even before Matiyasevi\v{c}'s contribution, a great deal of progress had been made by Davis, Putnam, and Robinson. It suffices to consider only natural number variables in the definition of a Diophantine set because of Lagrange's four-square theorem: For fixed $n$,
$$x^2+y^2+z^2+w^2=n\qquad x,y,z,w\in\Z$$ 
has a solution if and only if $n\in\N$, so any natural number quantifier may be replaced with four integer quantifiers. (Conversely, $n-m=x$ has a solution with $n,m\in\N$ iff $x\in\Z$.)

A Diophantine relation is a relation that is Diophantine as a subset of $\N^k$, and a function $f:\N^k\to\N$ is Diophantine if its graph $\{(\bar x,y)\mid y=f(\bar x)\}\subseteq\N^{k+1}$ is Diophantine. The simple arithmetic functions and relations $+,-,\times,<,\le,=$ are easily shown to be Diophantine, and conjunctions and disjunctions of Diophantine relations are Diophantine because of the equivalences:
$$p(\bar x)=0\wedge q(\bar x)=0\iff p(\bar x)^2+q(\bar x)^2=0$$
$$p(\bar x)=0\lor q(\bar x)=0\iff p(\bar x)q(\bar x)=0$$
Thus $\ne$ is also Diophantine since $x\ne y\leftrightarrow x<y\lor x>y$. (The complement of a Diophantine set is not necessarily Diophantine.) The notable omission from this list is the power function $x^y$, which is the key to showing that one can encode sequences of variable length as numbers. Given this, it would be possible to encode the solutions to an unbounded search problem, such as a halting computation, as the solution to a Diophantine equation.

This is the critical piece of the puzzle that Matiyasevi\v{c} solved, by showing that the power function $x^y$ is Diophantine.
\begin{theorem}[Matiyasevi\v{c}, 1970]
\label{matiya}
The function $f(x,y)=x^y$ is Diophantine. That is, there is a $k\in\N$ and a $k+3$-ary integer polynomial $p(x,y,w,\bar z)$ such that
$$x^y=w\iff \exists\bar z\in\N,\ p(x,y,w,\bar z)=0.$$
\end{theorem}
In this paper, we will present a formalized%
\footnote{The latest version of the proof is available as part of the Lean \texttt{mathlib} mathematics library, at:\\ \url{https://github.com/leanprover/mathlib/blob/master/number\_theory/pell.lean}.}
proof of theorem \ref{matiya}, in the Lean theorem prover. The proof naturally separates into two parts:
\begin{itemize}
\item The number theory: Developing the theory of the Pell equation, which features prominently in the proof.
\item The theory of Diophantine sets and functions, composition of Diophantine functions, and the proof that all the relevant functions in the proof are Diophantine.
\end{itemize}

The proof is based on a well-written exposition by Davis \cite{davis}. Section \ref{numbers} will discuss the number theory part of the proof, and section \ref{dioph} will discuss the engineering challenges that arise in proving that certain sets are Diophantine.
\section{The number theory}
\label{numbers}

It was already recognized before Matiyasevi\v{c} that the following ``Julia Robinson hypothesis'' would suffice to prove that the exponential function is Diophantine:
\begin{quote}
\emph{There exists a Diophantine set $D\subseteq\N^2$ such that:
\begin{itemize}
\item $(x,y)\in D\to y \le x^x$
\item For every $n$, there exists $(x,y)\in D$ such that $y > x^n$.
\end{itemize}}
\end{quote}
Thus it is really exponential \emph{growth rate} that is important, because the exponential function itself can be recognized as solving a system of congruences.

In Matiyasevi\v{c}'s original proof, the Fibonacci sequence was used for this purpose. In Davis's proof (which we follow), the Pell $x_n,y_n$ sequences are used instead, but the ideas are similar.

Fix an integer $d>0$ which is not a square, and consider the equation
$$x^2-dy^2=1.$$
This is obviously a Diophantine equation, and what makes it suitable for this application is that it has infinitely many solutions but they are spread out exponentially. In fact, there exists a ``fundamental solution'' $x_1,y_1$ such that every solution is given by
$$x_n+y_n\sqrt d=(x_1+y_1\sqrt d)^n$$
for some $n\in\N$. We will be interested in the special case $d=a^2-1$ for some $a>1$, in which case we can use the explicit fundamental solution $x_1=a$, $y_1=1$.

In order to formalize this kind of definition in Lean, we constructed the ring $\Z[\sqrt d]$ consisting of elements of the form $a+b\sqrt d$, where $a,b\in\Z$.

\subsection{The ring $\Z[\sqrt d]$}

\begin{lstlisting}
  structure zsqrtd (d : ℕ) := (re : ℤ) (im : ℤ)
  prefix `ℤ$√$`:100 := zsqrtd
\end{lstlisting}

This defines $\Z[\sqrt d]\simeq\Z\times\Z$ with the projection functions called "re" and "im", by analogy to the real and imaginary part functions of $\Z[i]$, although of course here both parts are real since $d$ is nonnegative (and sets up the notation "ℤ√d" for "zsqrtd d").

On this definition we can straightforwardly define addition and multiplication by their actions on the ``real'' and ``imaginary'' parts:

\begin{lstlisting}
  def add : ℤ√d → ℤ√d → ℤ√d
  | ⟨x, y⟩ ⟨x', y'⟩ := ⟨x + x', y + y'⟩

  def mul : ℤ√d → ℤ√d → ℤ√d
  | ⟨x, y⟩ ⟨x', y'⟩ := ⟨x * x' + d * y * y', x * y' + y * x'⟩
\end{lstlisting}

In order to recover the conventional ordering inherited from $\mathbb{R}$, we proceed in stages:

\begin{lstlisting}
  def sq_le (a c b d : ℕ) : Prop := c * a * a ≤ d * b * b

  def nonnegg (c d : ℕ) : ℤ → ℤ → Prop
  | (a : ℕ) (b : ℕ) := true
  | (a : ℕ) -[1+ b] := sq_le (b+1) c a d
  | -[1+ a] (b : ℕ) := sq_le (a+1) d b c
  | -[1+ a] -[1+ b] := false

  def nonneg : ℤ√d → Prop
  | ⟨a, b⟩ := nonnegg d 1 a b
\end{lstlisting}

The four-part relation "sq_le$\;a\;c\;b\;d$" expresses that $a\sqrt c\le b\sqrt d$ when $a,b,c,d$ are all nonnegative integers. In our application, either $c$ or $d$ will always be $1$, but it is convenient for the proof to have a ``symmetrized'' predicate for this to decrease the number of cases.

The generalized nonnegativity predicate "nonnegg $c$ $d$ $a$ $b$" asserts that $a\sqrt c+b\sqrt d\ge 0$, when $c,d\in\N$ and $a,b\in\Z$, by cases on whether $a,b$ are nonnegative or negative. This can then be used to construct the nonnegativity predicate "nonneg" on "ℤ√d", and then we may define $z\le w\iff \mathtt{nonneg}(w-z)$ when $z,w\in{}$"ℤ√d". 

From these definitions, it can be shown that "ℤ√d" is an archimedean linearly ordered commutative ring, although we must add the assumption that $d$ is a non-square to prove antisymmetry of $\le$, which in this context is equivalent to the assertion that $\sqrt d$ is irrational.

\begin{remark}
It may reasonably be asked why we did not define "ℤ√d" as a certain subset of $\mathbb{R}$ with the induced operations. The advantage of this approach is that all the operations here are computable, and the existence of an $a+b\sqrt d$ normal form for every element is built in to the definition. (Also, a definition of $\mathbb{R}$ with a square root function did not exist at the time of writing of the formalization.)
\end{remark}

\subsection{Pell's equation}
Most of the concepts surrounding Pell's equation translate readily to statements about elements of $\Z[\sqrt d]$. There is a ``conjugation'' operation $(a+b\sqrt d)^*=a-b\sqrt d$ on $\Z[\sqrt d]$, in terms of which $(x,y)$ is a solution to Pell's equation iff $N(x+y\sqrt d)=1$, where $N(z)=zz^*$ is the norm on $\Z[\sqrt d]$. Let $z_n=x_n+y_n\sqrt d$; then $z_n=z_1^n$ where $z_1=a+\sqrt d$, and the conclusion of Pell's theorem states that a solution to Pell's equation is equal to $z_n$ for some $n$.

\begin{remark}
In fact, we don't use the "ℤ√d" definition as the definition for the Pell sequences, because then we would lose the information that they are nonnegative integers. Instead, we use a mutual recursive definition and show thereafter that $z_n=x_n+y_n\sqrt d$.
\begin{lstlisting}
  mutual def xn, yn
  with xn : ℕ → ℕ
  | 0     := 1
  | (n+1) := xn n * a + d * yn n
  with yn : ℕ → ℕ
  | 0     := 0
  | (n+1) := xn n + yn n * a
\end{lstlisting}
\end{remark}

\begin{theorem}
If $N(z)=1$ and $z\ge 1$, then there exists $n\ge 0$ such that $z=z_n$.
\end{theorem}
\begin{proof}
Note that $z_n^*$ is the inverse of $z_n$ because $N(z_n)=1$. Also, $z_1>1$, so the sequence $z_n=z_1^n$ is monotonically increasing and unbounded. So if the theorem is false, then there exists $n$ such that $z_n<z<z_{n+1}$, so, multiplying by $z_n^*$, $1<zz_n^*<z_1$, so $w=zz_n^*$ also has unit norm and lies between $1$ and $z_1$.

Since inverses are conjugates, $1<w<z_1$ implies $1>w^*>z_1^*$ so $$0<w-w^*<z_1-z_1^*=2\sqrt d,$$
but if $w=a+b\sqrt d$ then $(a+b\sqrt d)-(a+b\sqrt d)^*=2b\sqrt d$, so $0<b<1$, which is impossible.
\end{proof}

But the key theorem here, Matiyasevi\v{c}'s contribution, is the following theorem:

\begin{theorem}\label{pell_dioph}
Let $a,k,x,y\in \N$ with $a>1$. Then 
$$x=x_k(a)\wedge y=y_k(a)\iff k\le y\land ((x=1\land y=0)\lor\varphi),$$
where
\begin{align*}
\varphi:=\exists u,v,s,t,b\in\N\quad
[&x^2-(a^2-1)y^2=1\ \wedge\\
&u^2-(a^2-1)v^2=1\ \wedge\\
&s^2-(b^2-1)t^2=1\ \wedge\\
&b>1\ \wedge\ b\equiv 1\pmod{4y}\ \wedge\ b\equiv a\pmod u\ \wedge\\
&v>0\ \wedge\ y^2\mid v\ \wedge\\
&s\equiv x\pmod u\ \wedge\ t\equiv k\pmod {4y}].
\end{align*}
\end{theorem}
The relevance of the complicated-looking right hand side is that all the relations there are easily seen to be Diophantine. (For example, $x\mid y$ iff $\exists k\,xk-y=0$ demonstrates that divisibility is Diophantine.) Thus the property of being the pair $x_n(a),y_n(a)$ is Diophantine in $x,y,n,a$. (For fixed $a,n$, this is trivial, and if we quantify over $n$ this is just the Pell equation, but the major advance is that now $n$ is a parameter.)

As the informal proof of this is laid out in \cite{davis} and the formal proof is online, we will omit the details here and refer interested readers to the formalization.
\section{Diophantine sets}
\label{dioph}
The informal proof contains many theorems in the style of Theorem \ref{pell_dioph}, where we have an equivalence between a target relation or function on one side, and a conjunction of disjunctions of quantified Diophantine predicates of Diophantine functions on the other, and will immediately conclude from such a theorem that the target relation is Diophantine. But anyone who has spent a significant time with formalization will recognize that these kinds of ``proof by inspection'' can be some of the worst offenders for disparity between the lengths of formal and informal proof. So as part of this proof, we were forced to develop a framework for dealing with such assertions.
\begin{lstlisting}
  def dioph {α : Type u} (S : set (α → ℕ)) : Prop :=
  ∃ {β : Type u} (p : poly (α ⊕ β)), ∀ (v : α → ℕ),
    v ∈ S ↔ ∃ t, p (v ⊗ t) = 0
\end{lstlisting}

We define a Diophantine set as a subset of $\N^\alpha$ (for some type $\alpha$) such that there exists a set of dummy variables $\beta$ and an integer polynomial $p$ with variables in $\alpha+\beta$ such that $v\in S$ iff there exists a $t\in\N^\beta$ such that $p(v,t)=0$.

The generality of types allows us to choose $\alpha$ and $\beta$ to be whatever type is convenient for specifying the problem, but for a systematic approach to translating formulas into proofs by composition we have to focus on the type "fin n" that contains $n$ elements, and "vector α n", the type of functions from "fin n" to $\alpha$. We generally want to prove a theorem of the form "dioph {v : α → ℕ | $\varphi$}" by a proof mirroring the structure of $\varphi$. For example, we can implement the divides relation and an existential quantifier like so:

\begin{lstlisting}
  theorem dvd_dioph {α : Type} {f g : (α → ℕ) → ℕ} :
    dioph_fn f → dioph_fn g → dioph (λ (v : α → ℕ), f v ∣ g v)

  theorem vec_ex1_dioph (n) {S : set (vector ℕ (succ n))}
    (d : dioph S) : dioph (λv : vector ℕ n, ∃x, S (x :: v))
\end{lstlisting}

Since the projections have to be referred by number, after a system of abbreviations, the result looks like the original formula written with de Bruijn indices (like the below for Theorem \ref{pell_dioph}). We hope to automate these proofs in the future.

\begin{lstlisting}
  D.1 D< D&0 D∧ D&1 D≤ D&3 D∧ ((D&2 D= D.1 D∧ D&3 D= D.0) D∨
   (D∃4 $\dol$ D∃5 $\dol$ D∃6 $\dol$ D∃7 $\dol$ D∃8 $\dol$
    D&7 D* D&7 D- (D&5 D* D&5 D- D.1) D* D&8 D* D&8 D= D.1 D∧ $\dots$
\end{lstlisting}
%
%

\end{document}